\documentclass[11pt]{amsart} 
\usepackage{mathrsfs}
\usepackage{amsthm}

\usepackage{amsfonts}
\usepackage{latexsym,amsmath,amssymb}

\usepackage[normalem]{ulem} 
\usepackage{xcolor}
\usepackage[colorlinks=true,linkcolor=blue]{hyperref}
\usepackage{indentfirst}
\renewcommand{\epsilon}{\varepsilon}

\newtheorem{theorem}{Theorem}[section]

\newtheorem{lemma}[theorem]{Lemma}

\newtheorem{defn}[theorem]{Definition}

\newtheorem{rem}[theorem]{Remark}
\newcommand{\bth}{\begin{theorem}}
	\newcommand{\ble}{\begin{lemma}}
		\newcommand{\bcor}{\begin{corr}}
			\newcommand{\bdeff}{\begin{defn}}
				\newcommand{\bprop}{\begin{proposition}}
				\newcommand{\ele}{\end{lemma}}
			\newcommand{\ecor}{\end{corr}}
		\newcommand{\edeff}{\end{defn}}
	
	\newcommand{\eprop}{\end{proposition}}

\renewcommand{\Pi}{\varPi}

\renewcommand{\epsilon}{\varepsilon}

\newcommand{\R}{{\mathbb R}}

\numberwithin{equation}{section}

\newtheorem{lem}{Lemma}[section]

\renewcommand{\P}{\partial}
\newcommand{\rD}{\mathrm{D}}

\newtheorem{definition}[lem]{Definition}

\thanks{The  authors were supported by  NSFC Grant No.12661037 and  Guangxi Natural Science Foundation No.2026GXNSFDA00640012.}
\title
{The symmetric maximal  surface  equation}

\author{Rongli Huang}
\address{School of Mathematics and Statistics, Guangxi Normal University,
	Guilin, Guangxi 541004, People's Republic of China, E-mail:ronglihuangmath@gxnu.edu.cn}
\author{Peihe Wang}
\address{School of Mathematical Sciences, Qufu Normal University, Qufu, 273165,
Shandong Province, China, E-mail:peihewang@hotmail.com}
\author{Hengyu Zhou}
\address{Chongqing Key Laboratory of Analytic Mathematics and Applications, College of Mathematics and Statistics, Chongqing University, Daxuecheng South Rd, 55, 401331 Chongqing, P. R. China,  E-mail:zhouhyu@cqu.edu.cn}
\begin{document}

\begin{abstract}
We establish the existence of smooth solutions to the symmetric maximal surface equation with degenerate boundary conditions. Moreover, we prove that these solutions maximize the associated area functionals. This result serves as the Lorentzian analogue of minimal graphs in hyperbolic spaces together with their associated area minimizing problem.

\end{abstract}
	\maketitle
	%\section{}
	%\subsection{}
	\let\thefootnote\relax

\section{Introduction}
Lorentzian geometry is the mathematical framework for general relativity and shares striking structural analogies with Riemannian geometry. Maximal and minimal surfaces, together with their governing equations, have been extensively studied. In this paper, we investigate one such parallel: the symmetric maximal surface equation and its associated hypersurfaces, which serve as the Lorentzian counterpart of the symmetric minimal surface equation proposed by Fouladgar and Simon \cite{FS20}.

The symmetric minimal surface equation
\[
-\mathrm{div}\left(\frac{\rD u}{\sqrt{1+|\rD u|^2}}\right)+\frac{m}{u\sqrt{1+|\rD u|^2}}=0,
\]
is central to Simon's construction \cite{Sim23} of stable minimal embedded hypersurfaces with arbitrary prescribed closed singular sets in high dimensions. Here $\rD u$ denotes the gradient of $u$ on $\Omega$, and $\mathrm{div}$ is the divergence operator on $\Omega$. This equation is the Euler--Lagrange equation of the functional
\[
\mathcal{F}_{\Omega}(u)=\int_{\Omega}u^m\sqrt{1+|\rD u|^2}\,\mathrm{dvol},
\]
for smooth domains $\Omega\subset\R^n$; its minimizers correspond to the area of rotationally symmetric hypersurfaces in $\R^{n+m+1}$, where $\mathrm{dvol}$ is the volume form on $\Omega$.\\
\indent In the Lorentzian setting, the analogous problem is to maximize the area-type functional
\begin{equation}\label{area:formula}
\int_{\Omega}u^m\sqrt{1-|\rD u|^2}\,\mathrm{dvol},
\end{equation}
for spacelike functions $u$ on $\Omega$. This functional corresponds to the area of a rotationally symmetric spacelike hypersurface in a certain Lorentzian manifold (see Subsection \ref{subsection:SME} for details). It gives rise to the following degenerate Dirichlet problem:
\begin{equation}\label{infinity:Dirichlet:problem}
\begin{cases}
\displaystyle
\mathrm{div}\left(\frac{\rD u}{\sqrt{1-|\rD u|^2}}\right)+\frac{m}{u\sqrt{1-|\rD u|^2}}=0, & \text{in } \Omega, \\[1.2ex]
u=0, & \text{on } \partial\Omega,
\end{cases}
\end{equation}
where $\Omega$ is a smooth bounded domain in an arbitrary Riemannian manifold $N$ with metric $\sigma$. Although if $m$ is the dimension of $\Omega$, the graph of $u$ is maximal with respect to the Lorentzian metric $r^2(\sigma-dr^2)$, the structure of \eqref{infinity:Dirichlet:problem} suggests a much closer resemblance to minimal graphs in hyperbolic spaces under the half-space model. Anderson \cite{Anderson} and Lin \cite{LFH} have studied the existence of minimal graphs in hyperbolic spaces and their associated area-minimizing properties.\\
\indent In the following, we obtain the first main result of this paper. 
\begin{theorem}\label{main:thm:A}
Fix $m>0$. Let $\Omega$ be a smooth bounded domain in a Riemannian manifold. Then the following hold:
\begin{enumerate}
\item There exists a unique function $u \in C^\infty(\Omega)\cap C(\bar{\Omega})$ solving \eqref{infinity:Dirichlet:problem} such that $u$ is positive and strictly spacelike, i.e. $|\rD u|<1$, in $\Omega$. Moreover, $|\rD u|=1$ on $\partial\Omega$.
\item The function $u$ uniquely maximizes the functional $\mathcal{A}_{\phi,\Omega}(u)$ with $\phi(r)=r$ (see \eqref{area:maximizing} below) among all Lipchitz strictly spacelike functions that are positive on $\Omega$ and vanish on $\partial\Omega$.
\end{enumerate}
\end{theorem}

The Dirichlet problem for the maximal surface equation with prescribed mean curvature,
\[
\mathrm{div}\!\left(\frac{\rD u}{\sqrt{1-|\rD u|^2}}\right)+F(x,u,\rD u)=0,
\]
has been extensively studied for smooth $F$. When $F=F(x,u)$, the problem was treated by Bartnik and Simon \cite{BS82}; later, the existence was established by Bartnik \cite{Bartnik84,Bartnik89} under some structural condition, and by Gerhardt \cite{gerhardt1983h} under a boundary acausal assumption. The main difficulty in \eqref{infinity:Dirichlet:problem} lies in the singular term $\frac{1}{u}$ as $u\to 0$---which also appears in the hyperbolic minimal graph setting \cite{LFH,Anderson}. Part (1) of Theorem \ref{main:thm:A} follows from Lemma \ref{lm:co:bound}, which provides a positive pointwise lower bound for approximating sequences, combined with an interior gradient estimate inspired by Bartnik \cite[Theorem 3.1]{Bartnik89}. 
Wang and Xiao obtained the existence of smooth, entire, strictly convex, spacelike, constant  curvature hypersurfaces with prescribed light-like directions in Minkowski space \cite{WX1} and the existence of convex entire spacelike hypersurfaces with constant  curvature in Minkowski space with prescribed boundary data at infinity \cite{WX2}.
For recent progress on other nonlinear equations arising in Lorentzian geometry, see \cite{BS25,BP26,Mani26,ma2026liouville,RWX1,RWX2} and references therein.\\
\indent Part (2) of Theorem \ref{main:thm:A} is derived from Lorentzian analogues of the area-minimizing problem in hyperbolic spaces \cite{LFH,Anderson}, Euclidean cones \cite{Rado,Tausch}, and conformal cones studied by Gao and the third author in \cite{GZ20,GZ21} under certain mean convex assumptions. More precisely, given a smooth bounded domain $\Omega$ and a strictly spacelike smooth function $\varphi$ on $\bar{\Omega}$, the area-type maximizing problem is to find   
\begin{equation}\label{area:maximizing}
\sup \left\{\mathcal{A}_{\phi,\Omega}(u) : u\in\mathcal{F}\right\}, \quad
\mathcal{A}_{\phi,\Omega}(u):=\int_{\bar{\Omega}}\phi(u)^m\sqrt{1-|\rD u|^2}\,\mathrm{dvol},
\end{equation}
where $\phi$ is a smooth positive function on $\mathbb{R}$, and $\mathcal{F}_{\varphi}$ denotes the set
\[
\mathcal{F}_{\varphi}=\{ u \in \mathrm{Lip}(\bar{\Omega}) : \mathrm{Lip}_{\Omega}(u)< 1,\; u = \varphi \text{ on } \partial\Omega,\; u>0\}.
\]
\indent In fact, for any spacelike function $u$, $\mathcal{A}_{\phi,\Omega}(u)$ is the Riemannian volume of its graph in the Lorentzian warped product $N\times \mathbb{R}$ endowed with the metric $\phi^2(r)(\sigma-dr^2)$ if $m$ is the dimension of $N$. The corresponding Dirichlet problem, which was obtained by Gerhardt \cite[Theorem 5.1]{gerhardt1983h} in much more general forms, is given by
\begin{equation} \label{general:Dirichlet:problem}
\begin{cases}
\operatorname{div}\!\left(\dfrac{\rD u}{\sqrt{1-|\rD u|^2}}\right)+\dfrac{\phi'(u)}{\phi(u)}\dfrac{m}{\sqrt{1-|\rD u|^2}}=0, & \text{in } \Omega, \\[1.2ex]
u=\varphi, & \text{on } \partial\Omega.
\end{cases}
\end{equation}
However, uniqueness is not guaranteed without additional assumptions. A monotonicity assumption on $\log\phi$ relates the area-type maximizing problem to the solution of \eqref{general:Dirichlet:problem} as follows.

\begin{theorem}\label{main:thm:B}
Let $\Omega$ be an smooth bounded domain in an $n$ dimensional Riemannian manifold, and let $\varphi$ be a smooth function satisfying $|\mathrm{D}\varphi|\le \kappa_0$ for some constant $\kappa_0<1$. Suppose that $(\log\phi)''\le 0$ on $\mathbb{R}$. Then there exists a unique spacelike function $u\in C^\infty(\Omega)\cap C(\overline{\Omega})$ solving the area maximizing problem in \eqref{area:maximizing}. Moreover, this function solves the Dirichlet problem in \eqref{general:Dirichlet:problem}.
\end{theorem}

In general Lorentzian manifolds, as established in \cite{Bartnik89}, the solution to the area maximizing problem among spacelike functions is smooth away from a one-dimensional singular set. Hence Theorem \ref{main:thm:B} provides a new nontrivial example in which the area-type maximizing surface is smooth. For related results, see \cite{Ecker86, Hong20, HY21, CPR21} for more details.  We remark that the positive constant $m$ appearing in \eqref{infinity:Dirichlet:problem} and \eqref{general:Dirichlet:problem} is arbitrary positive and is independent of the dimension of $N$.

This paper is organized as follows. In Section 2, we recall some preliminary facts about the symmetric maximal surface equation. Section 3 presents the existence and uniqueness of the radial solution by means of Banach's fixed point theorem, along with an asymptotic estimate. In Section 4, we derive an interior gradient estimate of \eqref{infinity:Dirichlet:problem} via the maximum principle and an auxiliary function; this estimate then yields part (1) of Theorem 1.1. In Section 5, we discuss area-type maximizing problems in our setting. In Subsection 5.1, we establish Theorem \ref{main:thm:B} (see Theorem \ref{main:thm:BA}). In Subsection 5.2, we obtain part (2) of Theorem \ref{main:thm:A} via the proof of Theorem \ref{main:thm:A2}.

\section{Preliminaries}

In this section, we collect some geometric facts about Lorentzian warped products of the following form.

\begin{definition}
Let $N$ be a Riemannian manifold with dimension $n$ and a metric $\sigma$, $\Omega\subset N$ be a smooth bounded domain, and $\phi$ be a positive smooth function on $\mathbb{R}$. A Lorentzian warped product is the product $\Omega\times \mathbb{R}$ endowed with the Lorentzian metric $\phi^2(r)(\sigma-dr^2)$, denoted by $\mathcal{L}_\phi$. The function $\phi$ is called the conformal factor. In particular, when $\phi\equiv 1$, it reduces to the standard Lorentzian cylinder and is denoted by $\mathcal{L}$.
\end{definition}

This definition generalizes the standard Lorentzian cylinder by introducing a conformal factor. Most of the facts presented here are analogous to those in \cite[Section 2]{GZ21} in the Riemannian setting.
\subsection{Lorentzian manifolds}

\begin{definition}
Let $M$ be a Lorentzian manifold and $\Sigma\subset M$ be a smooth hypersurface with unit normal vector $\vec{v}$. We say that $\Sigma$ is spacelike if $\langle \vec{v},\vec{v}\rangle=-1$.
\end{definition}
This implies that the induced metric on $\Sigma$ is non-degenerate and Riemannian. The mean curvature of $\Sigma$ with respect to $\vec{v}$ is
\[
\mathrm{H}=\mathrm{div}(\vec{v})=\langle e_i,\bar{\nabla}_{e_i}\vec{v}\rangle,
\]
where $\{e_1,\dots,e_n\}$ is any local orthonormal frame on $\Sigma$, and $\bar{\nabla}$ denotes the Levi-Civita connection of $M$.

\begin{definition}
A hypersurface $\Sigma$ is \emph{maximal} in $M$ if and only if its mean curvature vanishes identically on $\Sigma$.
\end{definition}

\subsection{Lorentzian cylinders}

We first collect some basic facts about Lorentzian cylinders with the metric $\sigma-dr^2$.

Let $u$ be a $C^2$ function on $\Omega$. If $|\mathrm{D}u|<1$, we say that $u$ is strictly spacelike. Its graph $\Sigma=\{(x,u(x)):x\in\Omega\}$ is then spacelike in the Lorentzian cylinder $\mathcal{L}$. The upward timelike normal vector of $\Sigma$ in $\mathcal{L}$ is given by
\begin{equation}
\vec{n}=\frac{\mathrm{D}u+\partial_r}{\sqrt{1-|\mathrm{D}u|^2}},\qquad \langle \vec{n},\vec{n}\rangle=-1.
\end{equation}

Let $\{x_1,\dots,x_n\}$ be local coordinates on $N$. We set $\langle \partial_{x_i},\partial_{x_j}\rangle=\sigma_{ij}$, and define the frame
\[
X_i=\partial_{x_i}+u_i\partial_r,\qquad i=1,\dots,n,
\]
on $\Sigma$. Here $\P_r$ is the vector field along $\R$ satisfying $\langle \P_r,\P_r\rangle =-1$. In local coordinates on $N$, the induced metric of $\Sigma$ satisfies
\[
g_{ij}=\langle X_i,X_j\rangle=\sigma_{ij}-u_i u_j,\qquad
(g^{ij})=(g_{ij})^{-1}=\sigma^{ij}+\frac{u^i u^j}{1-|\mathrm{D}u|^2},
\]
where $(\sigma^{ij})=(\sigma_{ij})^{-1}$. We write $u^i=\sigma^{ik}u_k$, and the gradient of $u$ is given by $\mathrm{D}u=u^i\partial_{x_i}$. The mean curvature of $\Sigma$ in $\mathcal{L}$ is
\begin{equation}\label{ML1}
\mathrm{H}=g^{ij}\langle \bar{\nabla}_{X_i}\vec{n},X_j\rangle
=\operatorname{div}\left(\frac{\mathrm{D}u}{\sqrt{1-|\mathrm{D}u|^2}}\right),
\end{equation}
where $\bar{\nabla}$ is the Levi-Civita connection of $\mathcal{L}$. We define an important quantity $v$ by 
\[
v:=\frac{1}{\sqrt{1-|\mathrm{D}u|^2}}=-\langle \vec{n},\partial_r\rangle.
\]
The Riemann curvature tensor $\bar{R}$ is defined by
\[
\bar{R}(X,Y)Z=\bar{\nabla}_X\bar{\nabla}_Y Z-\bar{\nabla}_Y\bar{\nabla}_X Z-\bar{\nabla}_{[X,Y]}Z,
\]
and Codazzi's equation for $\Sigma$ is
\begin{equation}\label{eq:ADE}
\bar{R}(X,Y,Z,\vec{n})=(\bar{\nabla}_X A)(Y,Z)-(\bar{\nabla}_Y A)(X,Z),
\end{equation}
where $A$ is the second fundamental form of $\Sigma$.
The Laplacian of $v$ on the spacelike hypersurface $\Sigma$ in $\mathcal{L}$ is given as follows.

\begin{lemma}\label{lm:laplacian}
Let $\Delta$ and $\nabla$ be the Laplacian and covariant derivative of $\Sigma$ in $\mathcal{L}$, respectively, and let $\mathrm{H}$ be the mean curvature of $\Sigma$ given in \eqref{ML1}. Then
\[
\Delta v=|A|^2 v-\langle \nabla \mathrm{H},\partial_r\rangle+v^3\,\mathrm{Ric}(\mathrm{D}u,\mathrm{D}u),
\]
where $\mathrm{Ric}$ denotes the Ricci curvature of the Riemannian manifold $N$, and $\mathrm{D}u$ is the gradient of $u$.
\end{lemma}
\begin{rem} The above result holds for any smooth hypersurfaces in Lorentzian cylinder. 
\end{rem}
\begin{proof}
Fix $p\in\Sigma$. Choose an orthonormal frame $\{e_1,\dots,e_n\}$ of $\Sigma$ at $p$ such that $\langle e_i,e_j\rangle=\delta_{ij}$ and $\nabla_{e_i}e_j(p)=0$. Recall that $h_{ij}=\langle \bar{\nabla}_{e_i}\vec{n},e_j\rangle$, where $\bar{\nabla}$ denotes the Levi-Civita connection of $\mathcal{L}$. Consequently, $\bar{\nabla}_{e_i}e_j=h_{ij}\vec{n}+\nabla_{e_i}e_j$.

The Laplacian of $v$ at $p$ is computed as follows:
\[
\begin{aligned}
\Delta v(p)&=-\nabla_{e_i}\langle \bar{\nabla}_{e_i}\vec{n},\partial_r\rangle(p)-\nabla_{\nabla_{e_i}e_i}v(p),\\
&=-\nabla_{e_i}\langle h_{ik}e_k,\partial_r\rangle(p),\\
&=|A|^2 v-h_{ik,i}\langle e_k,\partial_r\rangle(p),
\end{aligned}
\]
where $|A|=(\sum_{i,k}h_{ik}^2)^{1/2}$ denotes the norm of the second fundamental form. By \eqref{eq:ADE}, we have
\[
h_{ik,i}-h_{ii,k}=\bar{R}(e_i,e_k,e_i,\vec{n}).
\]
Combining the above identities yields
\[
\Delta v(p)=|A|^2 v-\langle \nabla \mathrm{H},\partial_r\rangle-\bar{R}(e_i,e_k\langle e_k,\partial_r\rangle,e_i,\vec{n}),
\]
where $\mathrm{H}=\sum_i h_{ii}$ is the mean curvature of $\Sigma$ given in \eqref{ML1}. The conclusion follows from $\langle e_k,\partial_r\rangle e_k=\partial_r-v\vec{n}$ and
\[
-\bar{R}(e_i,e_k\langle e_k,\partial_r\rangle,e_i,\vec{n})
=v\,\overline{\mathrm{Ric}}(\vec{n},\vec{n})
=v^3\,\mathrm{Ric}(\mathrm{D}u,\mathrm{D}u).
\]
Here $\bar{R}$ and $\overline{\mathrm{Ric}}$ denote the curvature tensor and Ricci curvature of the ambient Lorentzian cylinder, while $\mathrm{Ric}$ denotes the Ricci curvature of $N$.

Since $p$ is arbitrary, the proof is completed.
\end{proof}
\subsection{Lorentzian warped products}

Let $M$ and $M_f$ be two Lorentzian manifolds with the same underlying smooth manifold, endowed with metrics $g$ and $e^{2f}g$, respectively. Let $\Sigma$ be a fixed smooth oriented hypersurface, and denote its mean curvature in $M$ and $M_f$ by $\mathrm{H}$ and $\mathrm{H}_f$, respectively. Following the same derivation as in \cite[Lemma 3.1]{zhou19}, we have
\begin{equation}
\mathrm{H}_f=e^{-f}(\mathrm{H}+n\,\mathrm{d}f(\vec{v})),
\end{equation}
where the dimension of $M$ is $n+1$ and $\vec{v}$ is the unit normal vector of $\Sigma$ in $M$.

Recall that $\mathcal{L}_\phi$ is endowed with the metric $\phi^2(r)(\sigma-dr^2)$, while $\mathcal{L}$ has the metric $\sigma-dr^2$. From \eqref{ML1}, the mean curvature of $\Sigma$ in the Lorentzian warped product $\mathcal{L}_\phi$ is
\begin{equation}\label{eq:divergence}
\mathrm{H}_{\log\phi}=\frac{1}{\phi}\left(\mathrm{div}\left(\frac{\mathrm{D}u}{\sqrt{1-|\mathrm{D}u|^2}}\right)+\frac{\phi'(u)}{\phi(u)}\frac{n}{\sqrt{1-|\mathrm{D}u|^2}}\right).
\end{equation}
where $\mathrm{div}$ and $\rD$ is the divergence and the gradient with respect to the metric of $\Omega$. 
This yields the following lemma.

\begin{lemma}\label{lm:meancurvature}
A strictly spacelike graph of $u$ in the Lorentzian warped product $\mathcal{L}_\phi$ is maximal if and only if, on $\Omega$, $u$ satisfies
\begin{equation}\label{MLPhi}
\mathrm{div}\left(\frac{\mathrm{D}u}{\sqrt{1-|\mathrm{D}u|^2}}\right)+\frac{\phi'(u)}{\phi(u)}\frac{n}{\sqrt{1-|\mathrm{D}u|^2}}=0,
\end{equation}
where $n$ is the dimension of $N$. 
\end{lemma}
\subsection{Symmetric maximal hypersurfaces}\label{subsection:SME}

Indeed, solutions to \eqref{MLPhi} admit another geometric interpretation as symmetric maximal hypersurfaces. This serves as the Lorentzian counterpart to the symmetric minimal surfaces in Riemannian geometry studied by Fouladgar and Simon \cite{FS20}. For illustration, we henceforth set $\phi(r)=r$.

Let $\mathbb{R}^{m+1}$ denote the $(m+1)$-dimensional Euclidean space ($m\ge 1$) with coordinates $(y_1,\dots,y_{m+1})$ and the standard Euclidean metric
\[
dy_1^2+\cdots+dy_{m+1}^2.
\]
Consider the generalized Lorentzian manifold $N\times\mathbb{R}^{m+1}$ endowed with the Lorentzian metric
\[
\sigma-dy_1^2-\cdots-dy_{m+1}^2,
\]
denoted by $N\times_L\mathbb{R}^{m+1}$. Let $S^m$ be the unit sphere in $\mathbb{R}^{m+1}$. We define an embedded hypersurface $\Sigma\subset N\times_L\mathbb{R}^{m+1}$ by the map
\[
F:\Omega\times S^m\to N\times\mathbb{R}^{m+1},\qquad F(x,\omega)=(x,u(x)\omega).
\]
In general, $\Sigma_u=F(\Omega\times S^m)$ is a spacelike hypersurface in $N\times_L\mathbb{R}^{m+1}$ provided $|\mathrm{D}u|<1$ on $N$. The unit normal vector of $\Sigma_u$ at $(x,u(x)\omega)$ 
is given by
\begin{equation}\label{normal:vector}
\mathbf{n}=\frac{\mathrm{D}u+\omega}{\sqrt{1-|\mathrm{D}u|^2}},
\end{equation}
where $\mathrm{D}u$ denotes the gradient of $u$ on $N$, and $\omega$ is the unit normal vector of $S^m$ in $\mathbb{R}^{m+1}$. One verifies that $\mathbf{n}$ is timelike, with $\langle \mathbf{n},\mathbf{n}\rangle=-1$. In fact, the following result holds.

\begin{lemma}
The spacelike hypersurface $\Sigma_u$ is maximal in $N\times_L\mathbb{R}^{m+1}$ if and only if $u$ solves \begin{equation}\label{MLphi}
\mathrm{div}\left(\frac{\mathrm{D}u}{\sqrt{1-|\mathrm{D}u|^2}}\right)+\frac{\phi'(u)}{\phi(u)}\frac{m}{\sqrt{1-|\mathrm{D}u|^2}}=0,
\end{equation} with $|\mathrm{D}u|<1$. Here $m$ is independent of the dimension of $N$. 
\end{lemma}

\begin{rem}
This justifies calling \eqref{MLPhi} the symmetric maximal surface equation, in analogy with \cite{FS20}.
\end{rem}

\begin{proof}
It suffices to show that the mean curvature of $\Sigma_u$ in $N\times_L\mathbb{R}^{m+1}$ is
\begin{equation}\label{mean:curvature}
\mathrm{div}(\mathbf{n})
=\mathrm{div}\left(\frac{\mathrm{D}u}{\sqrt{1-|\mathrm{D}u|^2}}\right)
+\frac{1}{u(x)}\frac{m}{\sqrt{1-|\mathrm{D}u|^2}}.
\end{equation}
Fix a point $p_0=(x_0,u(x_0)\omega_0)\in\Sigma_u$. Let $\{x_1,\dots,x_n\}$ be local coordinates near $x_0$, and let $\{e_{n+1},\dots,e_{n+m}\}$ be a local orthonormal frame near $\omega_0$ on $S^m$. Then a local frame on $\Sigma_u$ near $p_0$ is given by $\{X_1,\dots,X_{n+m}\}$, where
\[
X_i=\frac{\partial}{\partial x_i}+\frac{\partial u}{\partial x_i}\omega,\quad i=1,\dots,n,\qquad
X_i=u e_i,\quad i=n+1,\dots,n+m.
\]
Set $g_{ij}=\langle X_i,X_j\rangle$. By definition,
\[
g_{ij}=0 \quad (i=1,\dots,n,\ j=n+1,\dots,n+m),
\]
\[
g_{ij}=-u^2\delta_{ij} \quad (i,j=n+1,\dots,n+m),
\]
and
\[
g_{ij}=\sigma_{ij}-u_i u_j \quad (i,j=1,\dots,n),
\]
where $\sigma_{ij}=\sigma(\partial_{x_i},\partial_{x_j})$ and $u_i=\partial u/\partial x_i$. Let $(g^{ij})=(g_{ij})^{-1}$, $(\sigma^{ij})=(\sigma_{ij})^{-1}$, and $u^i=\sigma^{ik}u_k$. Consequently,
\[
g^{ij}=0 \quad (i=1,\dots,n,\ j=n+1,\dots,n+m),
\]
\[
g^{ij}=-\frac{\delta^{ij}}{u^2} \quad (i,j=n+1,\dots,n+m),
\]
and
\[
g^{ij}=\sigma^{ij}+\frac{u^i u^j}{1-|\mathrm{D}u|^2} \quad (i,j=1,\dots,n).
\]

Let $\bar{\nabla}$ denote the covariant derivative on $N\times_L\mathbb{R}^{m+1}$. The divergence of $\mathbf{n}$ at $p_0=(x_0,u(x_0)\omega_0)$ is given by
\begin{equation}\label{st:A}
\mathrm{div}(\mathbf{n})(p_0)
=\sum_{i,j=1}^{n}g^{ij}\langle \bar{\nabla}_{X_i}\mathbf{n},X_j\rangle(p_0)
-\sum_{i,j=n+1}^{n+m}g^{ij}\langle \bar{\nabla}_{X_i}\mathbf{n},X_j\rangle(p_0),
\end{equation}
where $\mathrm{div}$ denotes the divergence with respect to the Lorentzian metric on $N\times_L\mathbb{R}^{m+1}$. The second term satisfies
\[
-\sum_{i,j=n+1}^{n+m}g^{ij}\langle \bar{\nabla}_{X_i}\mathbf{n},X_j\rangle(p_0)
=\frac{m}{u(x_0)\sqrt{1-|\mathrm{D}u|^2}},
\]
evaluated at $(x_0,u(x_0)\omega_0)$, where $m/u(x)$ is the mean curvature of the sphere of radius $u(x)$ in $\mathbb{R}^{m+1}$.

Since the restriction of the Lorentzian metric on $N\times_L\mathbb{R}^{m+1}$ to $N$ coincides with the Riemannian metric on $N$, we have
\[
\begin{aligned}
\sum_{i,j=1}^{n}g^{ij}\langle \bar{\nabla}_{X_i}\mathbf{n},X_j\rangle(p_0)
&=\sum_{i,j=1}^{n}g^{ij}\cdot\frac{1}{\sqrt{1-|\mathrm{D}u|^2}}
\left\langle\bar{\nabla}_{\partial_{x_i}}\mathrm{D}u,\partial_{x_j}\right\rangle\\
&=\mathrm{div}\left(\frac{\mathrm{D}u}{\sqrt{1-|\mathrm{D}u|^2}}\right).
\end{aligned}
\]
Substituting the above two identities into \eqref{st:A} yields \eqref{mean:curvature}. This completes the proof.
\end{proof}

\section{The radial solution of the symmetric maximal surface equation}

Throughout this section, $\alpha$ denotes a fixed positive constant. We consider the Dirichlet problem for the following nonlinear partial differential equation:
\begin{equation}\label{eq:pde}
\operatorname{div}\!\left(\frac{\mathrm{D}u}{\sqrt{1-|\mathrm{D}u|^2}}\right)
=-\frac{\alpha}{u}\frac{1}{\sqrt{1-|\mathrm{D}u|^2}},
\qquad x\in B_R(0),\quad u=0\text{ on }\partial B_R(0),
\end{equation}
where $B_R(0)\subset\mathbb{R}^n$ is the open ball of radius $R$ centered at the origin, $n\ge 1$, and $u>0$ in $B_R(0)$.

The main result of this section is as follows.

\begin{theorem}\label{thm:main:dke}
For any $n\ge 1$, $\alpha>0$, and $R>0$,  problem \eqref{eq:pde} admits a unique radially symmetric solution $u(x)=u(|x|)$ satisfying:
\begin{enumerate}
\item $u\in C^2[0,R)$, $u(r)>0$ for $r\in[0,R)$, $u'(0)=0$, $\lim_{r\rightarrow R^{-}}u(r)=0$;
\item $-1<u'(r)<0$ for $r\in (0,R)$, $\lim_{r\rightarrow R^{-}}u'(r)=-1$.
\end{enumerate}
\end{theorem}
\begin{rem}
The condition $u'(0)=0$ ensures that $u(x)\in C^2(B_R(0))$; smoothness then follows from \eqref{eq:pde} by Schauder estimates.
\end{rem}
\begin{proof}[Proof of Theorem \ref{thm:main:dke}]Our proof is divided into five steps. 
\subsection*{Step One: Reduction to an integral equation}. 
First we write the radial solution to \eqref{eq:pde} into a system of ordinary differential equations. \\
\indent Let $u(x)=u(r)$, where $r=|x|$. Then $\mathrm{D}u=u'(r)\frac{x}{r}$ and $|\mathrm{D}u|^2=u'(r)^2$ for $x\neq 0$. For radial functions,
\[
\operatorname{div}\!\left(u(r)\frac{x}{r}\right)=u'(r)+\frac{n-1}{r}u(r).
\]
Setting
\begin{equation}
v(r)=\frac{u'(r)}{\sqrt{1-u'(r)^2}},
\end{equation}
we obtain
\[
u'(r)=\frac{v(r)}{\sqrt{1+v(r)^2}},\qquad
\sqrt{1-u'(r)^2}=\frac{1}{\sqrt{1+v(r)^2}}.
\]
As a result, we conclude that the radial solution to equation \eqref{eq:pde} is equivalent to the first-order system of ordinary differential equations givens by 
\begin{equation}\label{eq:system}
\begin{cases}
u'(r)=\dfrac{v(r)}{\sqrt{1+v(r)^2}},\\[8pt]
v'(r)=-\dfrac{n-1}{r}\,v(r)-\dfrac{\alpha}{u(r)}\sqrt{1+v(r)^2},
\end{cases}
\end{equation}
with initial conditions $u(0)=u_0>0$ and $v(0)=0$, the latter corresponding to $u'(0)=0$.
\subsection*{Step Two: Scaling property} Suppose $(u(r),v(r))$ is a solution of system \eqref{eq:system} on $[0,R^*]$ with $u(R^*)=0$, $u(0)=u_0$ for some $R^*>0$. Of course we also have the condition $u'(0)=v(0)=0$. For any $\lambda>0$, define the functions by 
\[
u_\lambda(r)=\lambda\,u\!\left(\frac{r}{\lambda}\right),\qquad
v_\lambda(r)=v\!\left(\frac{r}{\lambda}\right). 
\]
A direct computation shows that $u_\lambda(r)$ and $v_{\lambda}(r)$
also solve \eqref{eq:system} on $[0,\lambda R^*]$, with $u_\lambda(0)=\lambda u_0$ and $u_\lambda(\lambda R^*)=0$.\\
\indent 
Taking $\lambda=R/R^*$, the function $u_R(r)=\lambda u(r/\lambda)$ satisfies $u_R(R)=0$ on $[0,R]$. By uniqueness of the initial value problem \eqref{eq:pde}, the solution for a given $R$ is also unique. We conclude that if the conclusion in Theorem \ref{thm:main:dke} holds for some $R^*>0$, then it holds for all $R>0$.

Therefore, it suffices to show that there exist some $u(0)=u_0>0$  such that the solution of the initial value problem reaches $u=0$ at some finite radius $R^*$.

\subsection*{Step Three: An existence result near $0$.}
Fix $u_0=1$ (to be adjusted later by scaling). Multiplying the equation for $v$ in system \eqref{eq:system} by $r^{n-1}$ gives
\[
\bigl(r^{n-1}v(r)\bigr)'=-\alpha\,r^{n-1}\frac{\sqrt{1+v(r)^2}}{u(r)}.
\]
Since $v(0)=0$, integrating from $0$ to $r$ yields
\[
r^{n-1}v(r)=-\alpha\int_0^r\frac{s^{n-1}\sqrt{1+v(s)^2}}{u(s)}\,ds,
\]
that is,
\begin{equation}\label{eq:integral}
v(r)=-\frac{\alpha}{r^{n-1}}\int_0^r\frac{s^{n-1}\sqrt{1+v(s)^2}}{u(s)}\,ds,\quad u(s)=u_0+\int_0^s\frac{v(t)}{\sqrt{1+v(t)^2}}\,dt.
\end{equation}

Define the integral operator $T$ by
\begin{equation}\label{eq:operator}
(Tv)(r)=-\frac{\alpha}{r^{n-1}}\int_0^r\frac{s^{n-1}\sqrt{1+v(s)^2}}{\,1+\displaystyle\int_0^s\frac{v(t)}{\sqrt{1+v(t)^2}}\,dt\,}\,ds.
\end{equation}
We seek a fixed point of $T$, i.e., a function $v$ satisfying $v=Tv$.

Take $u_0=1$ and define the weighted space
\[
X_\delta=\left\{v\in C[0,\delta]\,:\,v(0)=0,\;\|v\|_*:=\sup_{r\in[0,\delta]}\frac{|v(r)|}{r}\le M\right\},
\]
where $\delta>0$ and $M>0$ are constants to be chosen. The space $(X_\delta,\|\cdot\|_*)$ is completed.

For $v\in X_\delta$, we have $|v(t)|\le Mt$. Since $x\mapsto x/\sqrt{1+x^2}$ is $1$-Lipschitz, it follows that
\[
\left|\frac{v(t)}{\sqrt{1+v(t)^2}}\right|\le|v(t)|\le Mt,
\]
and hence
\[
\left|\int_0^s\frac{v(t)}{\sqrt{1+v(t)^2}}\,dt\right|\le\frac{Ms^2}{2}\le\frac{M\delta^2}{2}.
\]
Choose $\delta$ such that $\frac{M\delta^2}{2}\le\frac{1}{2}$; then
\begin{equation}\label{eq:u-lower}
u(s)=1+\int_0^s\frac{v(t)}{\sqrt{1+v(t)^2}}\,dt\ge\frac{1}{2}>0.
\end{equation}

Using \eqref{eq:u-lower}, for $v\in X_\delta$ we estimate
\[
|Tv(r)|\le\frac{\alpha}{r^{n-1}}\int_0^r\frac{s^{n-1}\sqrt{1+M^2s^2}}{1/2}\,ds
\le 2\alpha\sqrt{1+M^2\delta^2}\cdot\frac{r}{n}
=2\frac{\alpha}{n}\sqrt{1+M^2\delta^2}\,r.
\]
Thus $\|Tv\|_*\le 2\frac{\alpha}{n}\sqrt{1+M^2\delta^2}$. Choose $M\ge 4\alpha/n$, and then pick $\delta$ sufficiently small so that $2\frac{\alpha}{n}\sqrt{1+M^2\delta^2}\le M$; then $T$ maps $X_\delta$ into itself.

To see that $T$ is a contraction, let $v_1,v_2\in X_\delta$, and set $u_i(s)=1+\int_0^s v_i(t)/\sqrt{1+v_i(t)^2}\,dt$. Then
\[
|Tv_1(r)-Tv_2(r)|\le\frac{\alpha}{r^{n-1}}\int_0^r s^{n-1}\left|\frac{\sqrt{1+v_1(s)^2}}{u_1(s)}-\frac{\sqrt{1+v_2(s)^2}}{u_2(s)}\right|ds.
\]
Using the inequality
\[
\left|\frac{a_1}{b_1}-\frac{a_2}{b_2}\right|\le\frac{|a_1-a_2|}{b_1}+\frac{a_2\,|b_1-b_2|}{b_1 b_2},
\]
together with the fact that both $\sqrt{1+x^2}$ and $x/\sqrt{1+x^2}$ are $1$-Lipschitz, we obtain
\[
|\sqrt{1+v_1^2}-\sqrt{1+v_2^2}|\le|v_1-v_2|\le s\,\|v_1-v_2\|_*,
\]
and
\[
|u_1(s)-u_2(s)|\le\int_0^s|v_1(t)-v_2(t)|\,dt\le\frac{s^2}{2}\,\|v_1-v_2\|_*.
\]
Substituting these estimates (using $u_i\ge 1/2$ and $\sqrt{1+v_2^2}\le\sqrt{1+M^2\delta^2}$) yields
\[
\left|\frac{\sqrt{1+v_1^2}}{u_1}-\frac{\sqrt{1+v_2^2}}{u_2}\right|
\le 2s\,\|v_1-v_2\|_*+4\sqrt{1+M^2\delta^2}\,s^2\,\|v_1-v_2\|_*.
\]
Consequently,
\[
\|Tv_1-Tv_2\|_*\le\alpha\left(\frac{2}{n+1}\,\delta+\frac{4}{n+2}\sqrt{1+M^2\delta^2}\,\delta^2\right)\|v_1-v_2\|_*.
\]
Choosing $\delta$ sufficiently small so that the factor in parentheses is less than $1/2$, we conclude that $T$ is a contraction. By Banach's fixed-point theorem, the integral equation \eqref{eq:integral} admits a unique solution $v\in X_\delta$ on $[0,\delta]$. Correspondingly,
\[
u(r)=1+\int_0^r \frac{v(t)}{\sqrt{1+v(t)^2}}\,dt
\]
satisfies $u'(0)=0$ and $u>0$ on $[0,\delta]$, and $(u,v)$ is a solution of system \eqref{eq:system} on $[0,\delta]$. 
\subsection*{Step Four: Expansion near $0$.}
Indeed, there exists a $\delta>0$ such that $u'(r)<0$ and $v(r)<0$ on $(0,\delta]$. 
As $r\to 0^+$, continuity of $u$ and $v$ with $u(0)=u_0$ and $v(0)=0$ gives $u=u_0+O(r)$,  $v=O(r)$ and
\[
\frac{s^{n-1}\sqrt{1+v(s)^2}}{u(s)}=\frac{s^{n-1}}{u_0}(1+O(s)).
\]
From \eqref{eq:integral}, it follows that $
v(r)=-\frac{\alpha}{n}\frac{r}{u_0}+O(r^2)$
and consequently
\[
\frac{v(r)}{\sqrt{1+v(r)^2}}=-\frac{\alpha}{n}\frac{r}{u_0}+O(r^2).
\]
Integrating again gives
\[
u(r)=u_0-\frac{\alpha}{2n}\frac{r^2}{u_0}+O(r^3).
\]
 Since $u'(r)=\dfrac{v(r)}{\sqrt{1+v(r)^2}}$, a direct continuous verification shows that  $v(r)<0$, $u'(r)<0$ on $(0,\delta]$. From \eqref{eq:integral}, $v'(r)<0$ on $(0,\delta]$.\\
 Define 
 \[ 
 R^*:=\sup\{\delta:  \quad u(r)>0, v(r)<0 \,\text{is finite}, \quad \text{on}\quad (0,\delta]\}.\] 
\subsection*{Step 3: $R^*$ is finite.}
Suppose $R^*$ is infinity. We claim that only $\lim_{r\to\infty}u=0$ can happen.  Notice that $u(r)$ is strictly decreasing on $(0,+\infty)$. Otherwise, there exists $\beta>0$ such that $2\beta\ge u(r)\ge\beta$ for all sufficiently large $r$. Then it holds that 
\begin{equation}\label{det:uvb}
v'=-\frac{n-1}{r}v-\frac{\alpha}{u}\sqrt{1+v^2}\le -\frac{\alpha}{u}\le -\frac{\alpha}{2\beta},
\end{equation}
for  sufficiently large $r$ and \begin{equation}\label{det:stb} v(r)\le v(r_0)-\frac{\alpha}{2\beta}(r-r_0).
\end{equation} This implies that  $u'=v/\sqrt{1+v^2}\to-1$, so $u$ would become negative in finite time, contradicting $u\ge\beta>0$. Hence $\lim_{r\to\infty}u(r)=0$.

Fix $\varepsilon>0$. Then there exists $r_\varepsilon$ such that $u(r)<\varepsilon/2$ for all $r>r_\varepsilon$. Since $v=\sqrt{1+(v)^2}u'<0$, we have
\begin{equation}\label{eq:A}
v'=-\frac{n-1}{r}v-\frac{\alpha}{u}\sqrt{1+v^2}\le -\frac{\alpha}{2u}\sqrt{1+v^2}\le -\frac{\alpha}{\varepsilon}\sqrt{1+v^2}.
\end{equation}
Thus
\begin{equation}\label{eq:B}
\frac{v'}{\sqrt{1+v^2}}\le -\frac{\alpha}{\varepsilon}.
\end{equation}
Noting that $(\operatorname{arcsinh} v)'=v'/\sqrt{1+v^2}$, integration yields
\begin{equation}\label{eq:C}
\operatorname{arcsinh} v(r)\le \operatorname{arcsinh} v(r_\varepsilon)-\frac{\alpha}{\varepsilon}(r-r_\varepsilon).
\end{equation}
Hence
\begin{equation}\label{eq:D}
v(r)\leq\sinh\!\left(\operatorname{arcsinh} v(r_\varepsilon)-\frac{\alpha}{\varepsilon}(r-r_\varepsilon)\right)\to-\infty,
\end{equation}
as $r\to\infty$. Consequently, there exists $r_1>r_\varepsilon$ such that $|v(r)|>1$ for all $r>r_1$. Then $u'(r)=v/\sqrt{1+v^2}<-1/\sqrt{2}$ for $r>r_1$, which implies
\[
u(r)<u(r_1)-\frac{1}{\sqrt{2}}(r-r_1).
\]
Thus $u(r)<0$ for sufficiently large $r$, a contradiction. Therefore $R^*$ is finite.\\
\indent \indent By ODE theory and \eqref{eq:integral}, the solution can be uniquely continued along $(0,\delta]$, until case (i): either $u$ reaches $0$ or case (ii): $v(r)$ blows up as $r\rightarrow \delta^{-}$. Repeating the derivation in \eqref{eq:A}-\eqref{eq:D} yields  that case (i) implies that case (ii). Namely $v(r)$ always go to $-\infty$ as $r\to R^{*-}$. Since $u'(r)=v(r)/\sqrt{1+v(r)^2}$, this implies that $\lim_{r\to (R^*)^{-}}u'(r)=-1$ and $-1<u'(r)<0$ for all $r\in(0,R^*)$.  \\
\indent Finally we show that $\lim_{r\rightarrow R^{*-}}u=0$. Otherwise, there exists a constant $\beta>0$ such that $2\beta \geq u(r)\geq \beta>0$ on $(0,R^{*})$. From \eqref{det:uvb} and \eqref{det:stb}, it holds that 
\[ 
\frac{v'}{\sqrt{1+v^2}}\geq -\frac{c(R^*)}{\sqrt{1+v^2}}-\frac{\alpha}{\beta}\geq -c(R*)-\frac{\alpha}{\beta}
\]
on $(0,R^*)$. Here $c(R^*)$ is a fixed positive constant. For fixed $r_0\in (0, R^*)$, this gives that  
\begin{equation}
\operatorname{arcsinh}  v(r)\geq \operatorname{arcsinh} v(r_0)- (c(R*)+\frac{\alpha}{\beta})(r-r_0).
\end{equation}
Since $R^*$ is finite, this contradicts that $v(r)$ approaches to $-\infty$ as $r\to R^{*-}$. Hence it always holds that $\lim_{r\to R^{*-}}u=0$. This gives the existence of Theorem \ref{thm:main:dke} for some $R^*$. For $R^*>0$, ($u(0)=1$), the uniqueness of Theorem \ref{thm:main:dke} follows from \eqref{eq:system} for $u_0=1$. Hence we conclude Theorem \ref{thm:main:dke} for some $R^*$. \\
\indent By Step Two, we complete the proof of Theorem \ref{thm:main:dke} for any $R>0$. 
\end{proof}

\section{Existence part of Theorem \ref{main:thm:A}}

In this section, we prove part (1) of Theorem \ref{main:thm:A}, which is stated as follows.

\begin{theorem}\label{main:thm:A1}
Fix any $m>0$. Let $\Omega$ be a smooth bounded domain in a Riemannian manifold $N$ with metric $\sigma$. Then there exists a unique function $u\in C^\infty(\Omega)\cap C(\overline{\Omega})$ solving \eqref{infinity:Dirichlet:problem} such that $u$ is positive and strictly spacelike in $\Omega$. Moreover $|Du|=1$ on $\P\Omega$. 
\end{theorem}
The first step is to show that
\begin{lemma}\label{lm:step:one}
Let $\Omega$ be given as in Theorem \ref{main:thm:A1}, and let $c>0$ be a constant. Then there exists a unique strictly spacelike solution $u\in C^\infty(\Omega)\cap C(\overline{\Omega})$ to the Dirichlet problem
\begin{equation}\label{delta}
\begin{cases}
\mathrm{div}\!\left(\dfrac{\mathrm{D}u}{\sqrt{1-|\mathrm{D}u|^2}}\right)+\dfrac{1}{u}\dfrac{m}{\sqrt{1-|\mathrm{D}u|^2}}=0, & \text{in } \Omega,\\[1.2ex]
u=c, & \text{on } \partial\Omega.
\end{cases}
\end{equation}
\end{lemma}

\begin{proof}
Define a smooth positive function $h$ on $\mathbb{R}$ by
\begin{equation}\label{def:h}
h(r)=
\begin{cases}
\dfrac{1}{r}, & r\in\left(\dfrac{c}{2},\, \infty\right),\\[1.2ex]
\dfrac{4}{c}, & r\in\left(-\infty,\, \dfrac{c}{4}\right),\\[1.2ex]
\end{cases}
\end{equation}
Consider the Dirichlet problem
\begin{equation}\label{det}
\begin{cases}
\mathrm{div}\!\left(\dfrac{\mathrm{D}u}{\sqrt{1-|\mathrm{D}u|^2}}\right)+m h(u)\dfrac{1}{\sqrt{1-|\mathrm{D}u|^2}}=0, & \text{in } \Omega,\\[1.2ex]
u=c, & \text{on } \partial\Omega.
\end{cases}
\end{equation}
Since $c$ is constant and $h$ is smooth on $\mathbb{R}$, Theorem 5.1 in \cite{gerhardt1983h} yields a strictly spacelike solution $u\in C^\infty(\Omega)\cap C(\overline{\Omega})$ to \eqref{det}. Since $h>0$, we have
\[
\mathrm{div}\!\left(\frac{\mathrm{D}u}{\sqrt{1-|\mathrm{D}u|^2}}\right)\le 0,
\]
so $u\ge c$ in $\Omega$ by the maximum principle.  By the definition of $h$, it follows that $h(u)=1/u$ on $\Omega$. Hence $u$ is the desired solution to \eqref{delta}. Uniqueness follows from the maximum principle with the fact $(-\frac{1}{u})'\geq 0$. This completes the proof.
\end{proof}

Let $u_c\in C^\infty(\Omega)\cap C(\overline{\Omega})$ denote the solution to \eqref{delta} with $|\mathrm{D}u_c|<1$ in $\Omega$ and $u_c=c$ on $\partial\Omega$. Since $(-1/u)'\ge 0$, the maximum principle implies that $u_c\ge u_{c'}$ in $\Omega$ whenever $c\ge c'>0$. Now fix a positive decreasing sequence $\{c_i\}_{i=1}^\infty$ converging to zero. The monotone sequence $\{u_{c_i}\}_{i=1}^\infty$ converges to a nonnegative function $u\in C^1(\Omega)\cap C(\overline{\Omega})$ with $u=0$ on $\partial\Omega$. \\
\indent The remaining task is to show that $u$ is indeed the desired solution of Theorem \ref{main:thm:A1}. 

\begin{lemma}\label{lm:co:bound}
We have $u(x)>0$ for every $x\in\Omega$ and $|\mathrm{D}u|=1$ on $\partial\Omega$.
\end{lemma}

\begin{proof}
Fix $x_0\in\Omega$. Let $R$ be a fixed positive constant less than one. 

Let $n$ denote the dimension of $\Omega$.  Fix any positive constant $\alpha<m$. By Theorem \ref{thm:main:dke}, there exists a pair of smooth functions $(u(r),v(r))$ on $[0,R]$ satisfying
\begin{equation}\label{eq:middle:step}
\begin{cases}
u'(r)=\dfrac{v(r)}{\sqrt{1+v(r)^2}},\\[8pt]
v'(r)=-\dfrac{n-1}{r}v(r)-\dfrac{\alpha}{u(r)}\sqrt{1+v(r)^2},
\end{cases}
\end{equation}
with $u(0)=u_0>0$, $|u'(r)|<1$ on $(0,R)$, and $\lim_{r\to R^-}|u'(r)|=1$. Set $r(x)=\operatorname{dist}(x,x_0)$ on $B_R(x_0)$. Define
\[
h_{\lambda R}(x)=u_\lambda(r(x))=\lambda u\!\left(\frac{r(x)}{\lambda}\right),
\]
where $\lambda>0$ is a constant determined later. Since $u'(0)=0$, $h_{\lambda R}$ is a positive, strictly spacelike $C^2$ function on $B_{\lambda R}(x_0)$ with $h_{\lambda R}=0$ on $\partial B_{\lambda R}(x_0)$. From \eqref{eq:middle:step} and the scaling property in Step Two in the proof of Theorem \ref{thm:main:dke}, it follows that on $B_{\lambda R}(x_0)$,
\begin{equation}\label{eq:key}
\begin{aligned}
\mathrm{div}\!\left(\frac{\mathrm{D}h_{\lambda R}}{\sqrt{1-|\mathrm{D}h_{\lambda R}|^2}}\right)
&=\mathrm{div}\bigl(v(r(x))\,\mathrm{D}r(x)\bigr)\\
&=v_\lambda'(r(x))+v_\lambda(r(x))\,\Delta r(x)\\
&=\left(\Delta r(x)-\frac{n-1}{r(x)}\right)v_\lambda(r(x))
-\frac{\alpha}{h_{\lambda R}(x)}\sqrt{1+v_\lambda(r(x))^2}\\
&=-\frac{1}{h_{\lambda R}(x)}\frac{1}{\sqrt{1-|\mathrm{D}h_{\lambda R}|^2}}
\left(\alpha+h_{\lambda R}(x)|\mathrm{D}h_{\lambda R}(x)|\,O(r(x))\right),
\end{aligned}
\end{equation}
where $\mathrm{div}$ and $\Delta$ denote the divergence and Laplacian on $\Omega$, respectively. Here we used the standard expansion
\[
\Delta r=\frac{n-1}{r}+\frac{1}{3}\mathrm{Ric}(\mathrm{D}r,\mathrm{D}r)+O(r^2)
\]
in the Riemannian setting (see \cite{S-Yau}). By Theorem 3.1, $u(r)$ is decreasing, so $h_{\lambda R}(x)\le u_\lambda(0)=\lambda u_0$. Hence
\[
h_{\lambda R}(x)|\mathrm{D}h_{\lambda R}(x)|\,O(r(x))\le \lambda K u_0,
\]
where $K>0$ depends only on the Ricci curvature and the diameter of $\Omega$. Thus there exists $\lambda_0>0$ such that for any $\lambda\in(0,\lambda_0)$, $\lambda K u_0<m-\alpha$. From \eqref{eq:key}, for every $\lambda\in(0,\lambda_0)$ and every $x\in B_{\lambda R}(x_0)$,
\begin{equation}\label{det:A}
\mathrm{div}\!\left(\frac{\mathrm{D}h_{\lambda R}}{\sqrt{1-|\mathrm{D}h_{\lambda R}|^2}}\right)
>-\frac{m}{h_{\lambda R}}\frac{1}{\sqrt{1-|\mathrm{D}h_{\lambda R}|^2}}.
\end{equation}
Moreover, $h_{\lambda R}=0$ on $\partial B_{\lambda R}(x_0)$.

Now choose $\lambda$ so small that $B_{\lambda R}(x_0)\subset\subset\Omega$. For each $c>0$, we have $u_c(x)\ge c>0=h_{\lambda R}(x)$ on $\partial B_{\lambda R}(x_0)$. By the maximum principle, $u_c(x)>h_{\lambda R}(x)$ for every $c>0$. Hence $u(x)\ge h_{\lambda R}(x)>0$ for all $x\in B_{\lambda R}(x_0)$. Since $x_0\in\Omega$ was arbitrary, we obtain $u(x)>0$ for every $x\in\Omega$.

It remains to prove the boundary claim. Fix $y_0\in\partial\Omega$. Note that we may choose $\lambda_0$ uniformly so that the inequality holds for any $x_0\in\overline{\Omega}$ and any $R\in(0,1)$. Fix such a $\lambda_0$. Then there exists $\lambda_1>0$ with the following property: for any $\lambda<\lambda_1<\lambda_0$ and any $R\in(0,1)$, there is some $x_0\in\Omega$ such that the embedded ball $B_{\lambda R}(x_0)$ is tangent to $\partial\Omega$ at $y_0$. The function $h_{\lambda R}$ then satisfies \eqref{det:A} with $h_{\lambda R}=0$ on $\partial B_{\lambda R}(x_0)$. Therefore $u\ge h_{\lambda R}>0$ in $B_{\lambda R}(x_0)$ and $u(y_0)=h_{\lambda R}(y_0)=0$.

We argue by contradiction. Suppose $|\mathrm{D}u|(y_0)<1-c$ for some fixed $c>0$. Since $\lim_{y\to y_0}|\mathrm{D}h_{\lambda R}|=1$, we work along the geodesic $\gamma$ starting from $x_0$ pointing into $\Omega$. For any $x\in\gamma$ sufficiently close to $y_0$,
\[
u(x)=u(x)-u(y_0)=\bigl(|\mathrm{D}u|(y_0)+\varepsilon\bigr)\,d(x,y_0)
< \left(1-\frac{c}{2}\right)d(x,y_0)\le h_{\lambda R}(x),
\]
where $\varepsilon>0$ is sufficiently small. This contradicts $u(x)>h_{\lambda R}(x)$ on $B_{\lambda R}(x_0)$. Hence $|\mathrm{D}u|(y_0)=1$ for every $y_0\in\partial\Omega$.
\end{proof}
At this stage, we only know that $u(x)$ is $C^1$, positive and $|\mathrm{D}u|\le 1$ on $\Omega$, as the limit of $\{u_{c_i}\}_{i=1}^\infty$. We will show
the  no-light-segment result in Lemma \ref{det:LA} below that $u$ is in fact strictly spacelike, i.e. $|\mathrm{D}u|<1$ in $\Omega$. This will complete the existence part of Theorem \ref{main:thm:A1}.

\begin{lemma}\label{det:LA}
Fix $x_0\in\Omega$ and an embedded ball $B_R(x_0)\subset\subset\Omega$. Then there exists a positive constant $\mu$, depending only on $\operatorname{diam}(\Omega)$, $c_0$, $\mathrm{Ric}$, $n$,
$m$, and $x_0$, such that
\[
\sup_{B_R(x_0)} v\le \mu,
\]
where $v=1/\sqrt{1-|\mathrm{D}u|^2}$.
\end{lemma}

\begin{rem}
The proof follows the idea in \cite[Theorem 3.1]{Bartnik89}. We include it here for the convenience of the reader.
\end{rem}

\begin{proof}
Fix $x_0\in\Omega$ and an embedded ball $B_R(x_0)\subset\subset\Omega$. By Lemma \ref{lm:co:bound} and continuity of $u$, there exists $c_0>0$ such that $u\ge 2c_0$ on $B_R(x_0)$. Since $u$ is the limit of $\{u_{c_i}\}_{i=1}^\infty$, we may assume without loss of generality that $c_i<c_0/2$ for all $i$.

In what follows, let $\tilde u$ denote any $u_{c_i}$ for some $i$. Define an auxiliary function $f_{c_0}$ by
\[
f_{c_0}(r)=
\begin{cases}
n\log(r-c_0), & c_0<r\le r_0,\\[1.2ex]
n\left(\dfrac{r}{r_0-c_0}+\log (r_0-c_0)-\frac{r_0}{r_0-c_0}\right), & r\in(r_0,\infty).
\end{cases}
\]
The function $f_{c_0}$ is piecewise smooth; after mollification we still denote it by $f=f_{c_0}$, and it satisfies
\begin{equation}\label{key:estimates}
\frac{(1+f')}{\frac{1}{n-1}|f'|^2+f''(r)}
\le C\bigl((r_0-c_0)+(r_0-c_0)^2\bigr),
\end{equation}
where $C$ depends only on $n$.

We now apply the maximum principle as in the proof of \cite[Theorem 3.1]{Bartnik89}. Define
\[
\phi(v,r)=\operatorname{arcosh} v+f_{c_0}(r).
\]
Let $\Sigma_{c_0}$ denote the part of the graph of $\tilde u$ lying over $\Omega\times[c_0,\infty)$. By the choice of $c_0$, $\Sigma_{c_0}$ is nonempty.

Since $f(r)\to-\infty$ as $r\to c_0^+$, the function $\phi$ attains its maximum at some interior point $p\in\Sigma_{c_0}$. Below, $\Delta$ and $\nabla$ denote the Laplacian and covariant derivative on $\Sigma_{c_0}$. Set $\psi=\operatorname{arcosh} v$. At the point $p$, we have
\begin{equation}\label{key:equation}
\nabla\psi(p)+f'\nabla r(p)=0,\qquad
\Delta\psi(p)+f''(r)|\nabla r|^2+f'(r)\Delta r\le 0.
\end{equation}

From \eqref{delta}, the mean curvature of the graph of $\tilde u$ satisfies
\[
\mathrm{H}=\frac{1}{\tilde{u}}v,\qquad v=\frac{1}{\sqrt{1-|\mathrm{D}\tilde u|^2}}.
\]
where $h$ is given by \eqref{def:h}. 
By Lemma \ref{lm:laplacian},
\begin{equation}\label{eq:delta:rv}
\begin{aligned}
\Delta v
&\ge |A|^2v-\langle\nabla \mathrm{H},\partial_r\rangle+v^3\,\mathrm{Ric}(\mathrm{D}\tilde u,\mathrm{D}\tilde u)\\
&\ge |A|^2v+\frac{1}{\tilde{u}^2}v(v^2-1)-\frac{1}{\tilde{u}}\langle\nabla v,\partial_r\rangle-Cv^3\\
&\ge (1-\varepsilon_1)|A|^2v-C(\varepsilon_1,\mathrm{Ric},c_0)v^3,
\end{aligned}
\end{equation}
where $\varepsilon_1>0$ is a small constant to be determined later. Here we used $\langle \nabla r,\nabla r\rangle=v^2-1$ and the estimates
\begin{equation}\label{fact:ABC}
\begin{aligned}
|\nabla v|^2(p)&=\sum_{i=1}^n \langle \lambda_i e_i,\partial_r\rangle^2
\le \lambda_1^2(v^2-1),\\
|\nabla v|&\le \lambda_1 v\le \varepsilon_1 |A|^2v+\frac{1}{4\varepsilon_1}v^3,
\end{aligned}
\end{equation}
where $\lambda_i$ ($i=1,\dots,n$) are the eigenvalues of the second fundamental form with respect to an orthonormal basis $\{e_i\}$ at $T_p\Sigma$, i.e. $\bar{\nabla}_{e_i}\vec n(p)=\lambda_i e_i(p)$, and we take $|\lambda_1|=\max_i|\lambda_i|$.

Since $v=\cosh\psi$ and from \eqref{fact:ABC},
\[
\Delta v=\cosh\psi|\nabla\psi|^2+\sinh\psi\,\Delta\psi,\qquad
\frac{|\nabla v|}{\sqrt{v^2-1}}=|\nabla\psi|.
\]
Moreover, from \eqref{fact:ABC}, we have $|\lambda_1|\ge|\nabla\psi|$. Combining this with \eqref{eq:delta:rv} yields
\[
\Delta\psi\ge \frac{\cosh\psi}{\sinh\psi}\left((1-\varepsilon_1)|A|^2-|\nabla\psi|^2-C(\varepsilon_1,\mathrm{Ric},c_0)v^2\right).
\]
The mean curvature is $\mathrm{H}=\sum_i\lambda_i$. By the Schwarz inequality,
\[
\begin{aligned}
|A|^2&=\lambda_1^2+\sum_{i=2}^n\lambda_i^2
\ge \lambda_1^2+\frac{1}{n-1}(\mathrm{H}-\lambda_1)^2\\
&\ge \left(1+\frac{1}{n}\right)\lambda_1^2-\mathrm{H}^2
\ge \left(1+\frac{1}{n}\right)|\nabla\psi|^2-C(c_0)v^2.
\end{aligned}
\]
Choosing $\varepsilon_1=1/(n^2-1)$, we obtain
\[
\Delta\psi\ge \frac{\cosh\psi}{\sinh\psi}\left(\frac{1}{n-1}|\nabla\psi|^2-C(n,\mathrm{Ric},c_0)v^2\right).
\]
Recall that $\Delta r=\mathrm{H}v\ge -C(c_0)v^2$. Substituting this into \eqref{key:equation}, at the point $p$ we have
\[
\begin{aligned}
\left(\frac{1}{n-1}|f'(r)|^2+f''(r)\right)|\nabla r|^2
&=\frac{1}{n-1}|\nabla\psi|^2+f''(r)|\nabla r|^2\\
&\le C(n,\mathrm{Ric},c_0)v^2(1+f'(r))\\
&\le C(n,\mathrm{Ric},c_0)(|\nabla r|^2+1)(1+f'(r)).
\end{aligned}
\]
Using \eqref{key:estimates}, we get at $p$,
\[
|\nabla r|^2\le C(n,\mathrm{Ric},c_0)\,C(n)\bigl((r_0-c_0)+(r_0-c_0)^2\bigr)(|\nabla r|^2+1).
\]
Choose $r_0$ sufficiently close to $c_0$ so that
\[
C(n,\mathrm{Ric},c_0)\,C(n)\bigl((r_0-c_0)+(r_0-c_0)^2\bigr)\le \frac12.
\]
Then $|\nabla r|=v^2-1\le C_2$ at $p$, where $C_2$ depends only on $c_0,n,\mathrm{Ric}$.

On the other hand, on $\Sigma_{c_0}$, we have $C_4\leq f_{c_0}(r)\le C_3 $, where $C_3,C_4$ depend only on $\operatorname{diam}(\Omega)$ and $c_0$. Hence on the whole $\Sigma_{c_0}$,
\[
v\le \mu:=\operatorname{cosh}
\{(\operatorname{arcosh}(C_2+1)+C_3-C_4)\}
\]
Thus on $B_R(x_0)$,
\[
v\le \mu.
\]
The above estimate holds for every $\tilde u=u_{c_i}$ and every $i$. By the interior Schauder estimates and the convergence of $\{u_{c_i}\}$, the conclusion follows for $u$ on $B_R(x_0)$.
\end{proof}
We now establish uniqueness. Suppose $v$ is another solution of the Dirichlet problem \eqref{infinity:Dirichlet:problem} on $\Omega$. Since $u_c=c>v$ on $\partial\Omega$, the maximum principle together with $(-1/u)'\ge 0$ implies $u_c\ge v$ for every $c>0$. Letting $c\to 0$, we obtain $u\ge v$.

Conversely, for each $c>0$, set $v_c:=v+c$. From \eqref{infinity:Dirichlet:problem}, we have
\[
\mathrm{div}\!\left(\frac{\mathrm{D}v_c}{\sqrt{1-|\mathrm{D}v_c|^2}}\right)
=-\frac{m}{v_c-c}\frac{1}{\sqrt{1-|\mathrm{D}v_c|^2}}
<-\frac{m}{v_c}\frac{1}{\sqrt{1-|\mathrm{D}v_c|^2}}
\quad\text{in }\Omega.
\]
Since $u_c=v_c$ on $\partial\Omega$, the classical maximum principle yields $v_c\ge u_c$ in $\Omega$, i.e. $v+c\ge u_c$ in $\Omega$. Letting $c\to 0$ gives $v\ge u$ in $\Omega$. Hence $v\equiv u$, proving the uniqueness.
The above discussion completes the proof of Theorem~\ref{main:thm:A1} (part (1) of Theorem~\ref{main:thm:A}).

\section{Area maximizing problems}

 In this section, we study the area maximizing problem in \eqref{area:maximizing} stated as follows: to find 
\[
\sup_{u\in \mathcal{F}_\varphi}\mathcal{A}_{\phi,\Omega}(u),
\]
where $m$ is any positive number, $
\mathcal{A}_{\phi,\Omega}(u)=\int_{\Omega}\phi^m(u)\sqrt{1-|\mathrm{D}u|^2}\,\mathrm{dvol}$ and $
\mathcal{F}_\varphi=\{ u\in \mathrm{Lip}(\overline{\Omega}): |\mathrm{D}u|<1,\ u=\varphi\text{ on }\partial\Omega\}.
$

We first prove Theorem \ref{main:thm:B} (Theorem \ref{main:thm:BA}) for $\phi$ defined on $\mathbb{R}$, and then specialize to the case $\phi(r)=r$ on $(0,\infty)$ to obtain part (2) in Theorem \ref{main:thm:A}.

\subsection{Area maximizing problems in Lorentzian warped products}

For the reader's convenience, we restate Theorem \ref{main:thm:B} as follows.

\begin{theorem}\label{main:thm:BA}
Let $\Omega$ be a smooth bounded domain in an $n$ dimensional Riemannian manifold, and let $\varphi$ be a smooth function satisfying $|\mathrm{D}\varphi|\le \kappa_0$ for some constant $\kappa_0<1$. Suppose that $(\log\phi)''\le 0$ on $\mathbb{R}$. Then there exists a unique spacelike function $u\in C^\infty(\Omega)\cap C(\overline{\Omega})$ solving the area maximizing problem in \eqref{area:maximizing}. Moreover, this function solves the Dirichlet problem in \eqref{general:Dirichlet:problem}.
\end{theorem}

\begin{proof}
Since $\phi$ is smooth on $\mathbb{R}$, by \cite[Theorem 5.1]{gerhardt1983h} there exists a strictly spacelike solution $u_0\in C^\infty(\Omega)\cap C(\overline{\Omega})$ to the Dirichlet problem
\begin{equation}\label{general:Dirichlet:problem:B}
\begin{cases}
\mathrm{div}\!\left(\dfrac{\mathrm{D}u}{\sqrt{1-|\mathrm{D}u|^2}}\right)
+\dfrac{\phi'(u)}{\phi(u)}\dfrac{m}{\sqrt{1-|\mathrm{D}u|^2}}=0, & \text{in } \Omega,\\[1.2ex]
u=\varphi, & \text{on } \partial\Omega.
\end{cases}
\end{equation}
Because $(\log\phi)''\le 0$, $u_0$ is the unique solution to \eqref{general:Dirichlet:problem:B} in $C^\infty(\Omega)\cap C(\overline{\Omega})$.

It remains to show that for every $v\in\mathcal{F}_\varphi$,
\[
\mathcal{A}_{\phi,\Omega}(v)\le \mathcal{A}_{\phi,\Omega}(u_0),
\]
with equality if and only if $v=u_0$.\\
\indent Let $\tilde{\phi}$ be the smooth positive function $\phi^{\frac{n}{m}}$. It holds that 
\begin{equation}\label{remark:A}
\dfrac{\tilde{\phi}'}{\tilde{\phi}}n=\dfrac{\phi'}{\phi}m
\end{equation}
Consider the Lorentzian warped product $\Omega\times \mathbb{R}$ equipped with the metric $\tilde{\phi}^2(r)(\sigma-dr^2)$, written as $L_{\tilde{\phi}}$.

Fix $t\in\mathbb{R}$, and let $u_t$ denote the solution of the Dirichlet problem
\begin{equation}\label{S-SME-t}
\begin{cases}
\mathrm{div}\!\left(\dfrac{\mathrm{D}u}{\sqrt{1-|\mathrm{D}u|^2}}\right)
+(\log\phi)'(u(x))\dfrac{m}{\sqrt{1-|\mathrm{D}u|^2}}=0, & \text{in } \Omega,\\[1.2ex]
u=\varphi+t, & \text{on } \partial\Omega.
\end{cases}
\end{equation}
Since $(\log\phi)''\le 0$, the solution $u_t$ is unique and depends continuously on $t$. Let $\operatorname{gr}(u_t)$ be the graph of $u_t$.The family $\{\operatorname{gr}(u_t)\}_{t\in\mathbb{R}}$ forms a smooth foliation of $\overline{\Omega}\times\mathbb{R}$. Define a smooth timelike unit vector field $X$ by setting $X=\vec v_t$ on $\operatorname{gr}(u_t)$ for each $t\in\mathbb{R}$, where $\vec v_t$ is the timelike normal vector of $\operatorname{gr}(u_t)$ in $L_{\tilde{\phi}}$.  By Lemma \ref{lm:meancurvature} and \eqref{remark:A}, $gr(u_t)$ is maximal in $L_{\tilde{\phi}}$, 
\begin{equation} \operatorname{\tilde{div}}(\vec{v}_t)=\mathrm{H}_{\log\tilde{\phi}}=0
\end{equation}
where $\operatorname{\tilde{div}}$ is the divergence of $L_{\tilde{\phi}}$. 
Let $\xi$ denote the Lorentzian volume form of $L_{\tilde{\phi}}$. By the definition of divergence (see \cite[P423]{Lee13}), the $(n-1)$-form $X\lrcorner  \xi$ is closed and satisfies
\[
d(X\lrcorner \xi)=\operatorname{\tilde{div}}(X)\xi=\operatorname{\tilde{div}}(\vec{v}_t)\xi=0,
\]
.

Let $E$ be the domain enclosed by the graphs of $v$ and $u_0$, and let $N$ denote the future timelike normal of $\operatorname{gr}(v)$. Applying Stokes' theorem it  gives
\[
\begin{aligned}
0=\int_E d(X\lrcorner \xi)
&=\int_{\operatorname{gr}(v)} X\lrcorner \xi
-\int_{\operatorname{gr}(u_0)} X\lrcorner \xi\\
&=\int_{\operatorname{gr}(v)}\langle X,N\rangle\,d\mathcal{H}^n
-\int_{\operatorname{gr}(u_0)}\langle \vec v_0,\vec v_0\rangle\,d\mathcal{H}^n,
\end{aligned}
\]
where $\mathcal{H}^n$ denotes $n$-dimensional Hausdorff measure and $\vec v_0$ is the timelike normal of $\operatorname{gr}(u_0)$. Since $\langle \vec v_0,\vec v_0\rangle=-1$ and $\langle X,N\rangle\le -1$ on $\operatorname{gr}(v)$, we conclude that
\begin{equation}\label{AB}
\mathcal{A}_{\phi,\Omega}(u_0)\ge \mathcal{A}_{\phi,\Omega}(v),
\end{equation}
with equality if and only if $v=u_0$ in $\Omega$. The proof is completed.
\end{proof}

\subsection{Area maximizing problems when $\phi(r)=r$}

In this subsection, we prove the following result.

\begin{theorem}\label{main:thm:A2}
Let $u_0$ be the unique solution from Theorem \ref{main:thm:A1} to the Dirichlet problem
\[
\begin{cases}
\mathrm{div}\!\left(\dfrac{\mathrm{D}u}{\sqrt{1-|\mathrm{D}u|^2}}\right)
+\dfrac{1}{u(x)}\dfrac{m}{\sqrt{1-|\mathrm{D}u|^2}}=0, & \text{in } \Omega,\\[1.2ex]
u=0, & \text{on } \partial\Omega.
\end{cases}
\]
Then $u_0$ uniquely solves the maximizing problem
\[
\sup_{v\in\mathcal{F}}\mathcal{A}_{\phi,\Omega}(v)
\]
for $\phi(r)=r$, where
\[
\mathcal{F}=\{v\in \mathrm{Lip}(\overline{\Omega}): |\mathrm{D}v|<1,\ v|_{\partial\Omega}=0,\ v|_{\Omega}>0\}.
\]
\end{theorem}

\begin{rem}
In this setting,
\[
\mathcal{A}_{r,\Omega}(v)=\int_{\Omega}v^m\sqrt{1-|\mathrm{D}v|^2}\,\mathrm{dvol}.
\]
By Theorem \ref{main:thm:A1}, $u_0\in\mathcal{F}$. The conclusion states that
\[
\int_{\Omega}u_0^m\sqrt{1-|\mathrm{D}u_0|^2}\,\mathrm{dvol}
=
\sup_{v\in\mathcal{F}}\int_{\Omega}v^m\sqrt{1-|\mathrm{D}v|^2}\,\mathrm{dvol}.
\]
Together with Theorem \ref{main:thm:A1}, this yields Theorem \ref{main:thm:A}.
\end{rem}

\begin{proof}
From the proof of Theorem \ref{main:thm:A1}, $u_0$ is the $C^0$-limit of the monotone sequence $\{u_{c_i}\}_{i=1}^\infty$, where $\{c_i\}_{i=1}^\infty$ is a positive sequence tending to $0$, and each $u_{c_i}$ satisfies
\[
\begin{cases}
\mathrm{div}\!\left(\dfrac{\mathrm{D}u}{\sqrt{1-|\mathrm{D}u|^2}}\right)
+\dfrac{1}{u(x)}\dfrac{m}{\sqrt{1-|\mathrm{D}u|^2}}=0, & \text{in } \Omega,\\[1.2ex]
u=c_i, & \text{on } \partial\Omega.
\end{cases}
\]

For any $v\in\mathcal{F}$, set $v_{c_i}=v+c_i$. Since $v$ is continuous on $\overline{\Omega}$, for each fixed $c_i>0$ there exists a positive constant $\beta_i$ (depending on $c_i$ and $v$) such that $v_{c_i}\ge \beta_i$. Define a new positive function $\phi_i$ on $\mathbb{R}$ by
\[
\phi_i(r)=
\begin{cases}
r, & r>\dfrac{\beta_i}{2},\\[1.2ex]
\dfrac{\beta_i}{4}, & r<\dfrac{\beta_i}{4}.
\end{cases}
\]
After mollification if necessary, we may assume $\phi_i$ is smooth and satisfies $(\log\phi_i)''\le 0$ on $\mathbb{R}$. Consequently, $u_{c_i}$ also solves
\[
\begin{cases}
\mathrm{div}\!\left(\dfrac{\mathrm{D}u}{\sqrt{1-|\mathrm{D}u|^2}}\right)
+\dfrac{\phi_i'(u)}{\phi_i(u)}\dfrac{m}{\sqrt{1-|\mathrm{D}u|^2}}=0, & \text{in } \Omega,\\[1.2ex]
u=c_i, & \text{on } \partial\Omega.
\end{cases}
\]

By the definition of $\phi_i$, we have
\[
\mathcal{A}_{r,\Omega}(u_{c_i})=\mathcal{A}_{\phi_i,\Omega}(u_{c_i}),\qquad
\mathcal{A}_{r,\Omega}(v_{c_i})=\mathcal{A}_{\phi_i,\Omega}(v_{c_i}).
\]
Since $u_{c_i}=v_{c_i}$ on $\partial\Omega$, Theorem \ref{main:thm:B} gives
\[
\mathcal{A}_{\phi_i,\Omega}(u_{c_i})\ge \mathcal{A}_{\phi_i,\Omega}(v_{c_i})
\]
for every $i$. Hence $\mathcal{A}_{r,\Omega}(u_{c_i})\ge \mathcal{A}_{r,\Omega}(v_{c_i})$. Letting $i\to\infty$ yields
\[
\lim_{i\to\infty}\mathcal{A}_{r,\Omega}(u_{c_i})=\mathcal{A}_{r,\Omega}(u_0),
\qquad
\lim_{i\to\infty}\mathcal{A}_{r,\Omega}(v_{c_i})=\mathcal{A}_{r,\Omega}(v).
\]
Therefore,
\[
\mathcal{A}_{r,\Omega}(u_0)\ge \mathcal{A}_{r,\Omega}(v).
\]
for every $v\in\mathcal{F}$. For uniqueness, suppose $u'\in\mathcal{F}$ satisfies $\mathcal{A}_{r,\Omega}(u')=\mathcal{A}_{r,\Omega}(u_0)$. Then $u'\in C^\infty(\Omega)\cap C(\overline{\Omega})$ and solves \eqref{infinity:Dirichlet:problem}. By Theorem \ref{main:thm:A}, $u'=u_0$. It completes the proof.
\end{proof}

\vspace{5mm}
\bibliographystyle{alpha} %数学常用样式，可换plain,alpha
\bibliography{Refs.bib}
\end{document}